\documentclass[10pt]{amsart}
\usepackage{amsmath,amsfonts,amsthm,amssymb}
\usepackage{graphicx}
\usepackage{epstopdf}

\theoremstyle{plain}
\newtheorem{theorem}{Theorem}[section]

\newtheorem{lemma}[theorem]{Lemma}
\newtheorem*{de-lemma}{Lemma}

\theoremstyle{remark}

\theoremstyle{definition}

\newcommand{\dd}{\mathrm{d}}
\newcommand{\R}{\mathbb{R}}\newcommand{\N}{\mathbb{N}}

\DeclareMathOperator{\supp}{supp}

\DeclareMathOperator{\ee}{{\bf e}}
\DeclareMathOperator{\VV}{{\bf V}}
\begin{document}

\title[Double layered solutions to the extended Fisher-Kolmogorov P.D.E.]{Double layered solutions to the extended Fisher-Kolmogorov P.D.E.}

\author{Panayotis Smyrnelis} \address[P.~ Smyrnelis]{Institute of Mathematics,
Polish Academy of Sciences, 00-656 Warsaw, Poland, and Basque Center for Applied Mathematics, 48009 Bilbao, Spain}
\email[P. ~Smyrnelis]{psmyrnelis@impan.pl, psmyrnelis@bcamath.org}
%\thanks{}

\date{}

\maketitle
\begin{abstract}
We construct double layered solutions to the extended Fisher-Kolmogorov P.D.E., under the assumption that the set of minimal heteroclinics of the corresponding O.D.E. satisfies a separation condition. The aim of our work is to provide for the extended Fisher-Kolmogorov equation, the first examples of two-dimensional minimal solutions, since these solutions play a crucial role in phase transition models, and are closely related to the De Giorgi conjecture.

\end{abstract}

\section{Introduction and Statements}
We consider the extended Fisher-Kolmogorov P.D.E.
\begin{equation}\label{system}
\Delta^2 u(t,x) -\beta \Delta u(t,x)+ \nabla W(u(t,x))=0, \ u:\R^2\to\R^m \ m\geq 1, \ \beta> 0, \ (t,x)\in\R^2,
\end{equation}
where $W$ is a function such that
\begin{subequations}\label{w13}
\begin{equation}\label{w1}
\text{$W\in C^{2}(\R^m; \R)$ is nonnegative, and has exactly $2$ zeros $a^-$ and $a^+$,}
\end{equation}
\begin{equation}\label{w2}
\text{$\nabla^2 W(u)(\nu,\nu)\geq c$, $\forall u\in\R^m$: $|u-a^\pm|\leq r$, $\forall \nu \in \R^m$: $|\nu|=1$, with $r,c>0$},
\end{equation}
\begin{equation}\label{w3}
\liminf_{|u|\to\infty} W(u)>0.
\end{equation}
\end{subequations}
That is, $W$ is a double well potential \eqref{w1}, with nondegenerate minima \eqref{w2}, satisfying moreover the standard asymptotic condition \eqref{w3} to ensure the boundedness of finite energy orbits.
To clarify the notation, we point out that $\nabla W(u(t,x))$ is the gradient of $W$ evaluated at $u(t,x)$, while $\nabla^2W(u)(\nu,\nu)$ stands for the quadratic form $\sum_{i,j=1}^m\frac{\partial^2 W(u)}{\partial u_i \partial u_j}\nu_i \nu_j$, $\forall u=(u_1,\ldots,u_m)\in\R^m$, $\forall \nu=(\nu_1,\ldots,\nu_m)\in\R^m$. We also denote respectively by $|\cdot|$ and $\cdot$, the Euclidean norm and inner product. Finally, given smooth maps $u:\R^2\to \R^m$, and $\phi:\R^2\to\R^m$, we set
\begin{itemize}
\item $|\nabla u|^2:=\sum_{i=1}^2\big|\frac{\partial u}{\partial x_i}\big|^2$, 
\item $|\nabla^2 u|^2:=\sum_{i,j=1}^2\big|\frac{\partial^2 u}{\partial x_i\partial x_j}\big|^2$, 
\item $\nabla u\cdot\nabla\phi:=\sum_{i=1}^2\big(\frac{\partial u}{\partial x_i}\cdot\frac{\partial \phi}{\partial x_i}\big)$,
\item $\nabla^2 u\cdot \nabla^2 \phi:=\sum_{i,j=1}^2\big(\frac{\partial^2 u}{\partial x_i\partial x_j}\cdot\frac{\partial^2 \phi}{\partial x_i\partial x_j}\big)$.
\end{itemize}

In the scalar case ($m=1$), by taking the Allen-Cahn potential $W(u)=\frac{1}{4}(u^2-1)^2$, 
we obtain the standard Fisher-Kolmogorov O.D.E.
\begin{equation}\label{efk}
\frac{\dd^4 u}{\dd x^4}-\beta u''+u^3-u=0, \ u:\R\to\R,
\end{equation}
which was proposed in 1988 by Dee and van Saarloos \cite{dee} as a higher-order model equation for bistable systems. Equation \eqref{efk} has been extensively studied by different methods: topological shooting methods, Hamiltonian methods, variational methods, and methods based on the maximum principle (cf. \cite{bon1}, \cite{pel1}, and the references therein). In these monographs, a systematic account is given of the different kinds of orbits obtained for O.D.E. \eqref{efk}, which has a considerably richer structure than second order phase transition models.

The existence of heteroclinic orbits of \eqref{efk} via variational arguments was investigated for the first time by L. A. Peletier, W. C. Troy
and R. C. A. M. VanderVorst \cite{pel2}, and W. D. Kalies, R. C. A. M. VanderVorst \cite{kal}. In the vector case $m\geq 1$, we established \cite{ps} the existence of minimal heteroclinics for a large class of fourth order systems, including the O.D.E.:
\begin{equation}\label{ode0}
\frac{\dd^4 u}{\dd x^4}(x)-\beta u''(x)+\nabla W(u(x)), \ u:\R\to\R^m\  (m\geq 1),\ \beta>0, \  x\in\R,
\end{equation}
with a double well potential $W$ as in \eqref{w13}. By definition, a heteroclinic orbit is a solution $e\in C^4(\R;\R^m)$ of \eqref{ode0} such that
$\lim_{x\to\pm\infty}(e(x),e'(x),e''(x),e'''(x))=(a^\pm,0,0,0)$ in the phase-space. A heteroclinic orbit is called \emph{minimal} if it is a minimizer of the Action functional associated to \eqref{ode0}:
\begin{equation}\label{JJ}
J_I(u):=\int_I\Big[\frac{1}{2}| u''(x)|^2+\frac{\beta}{2}|u'(x)|^2+W(u(x))\Big]\dd x, \ I\subset \R,
\end{equation}
in the class $A:=\{ u\in H_{\rm loc}^{2}(\R;\R^m):\ \lim_{x\to\pm\infty}u(x)=a^\pm\}$, i.e. if $J_\R(e)=\min_{u\in A} J_\R(u)=:J_{\mathrm{min}}$. Assuming \eqref{w13}, we know that there exists at least one \emph{minimal} heteroclinic orbit $e$ (cf. \cite{ps}). In addition, since the minima $a^\pm$ are nondegenerate, the convergence to the minima $a^\pm$ is exponential for every minimal heteroclinic $e$, i.e.
\begin{subequations}\label{expest}
\begin{equation}
|e(x)-a^-|+|e'(x)|+|e''(x)|+|e'''(x)|+|e''''(x)|\leq K e^{kx}, \forall x\leq 0,
\end{equation}
\begin{equation}
|e(x)-a^+|+|e'(x)|+|e''(x)|+|e'''(x)|+|e''''(x)|\leq K e^{-kx}, \forall x\geq 0,
\end{equation}\end{subequations}
where the constants $k,K>0$ depend on $e$ (cf. \cite[Proposition 3.4.]{ps}).
Clearly, if $x\mapsto e(x)$ is a heteroclinic orbit, then the maps
\begin{equation}\label{transl}
x\mapsto e^T(x):=e(x-T),  \forall T\in\R,
\end{equation}
obtained by translating $x$, are still heteroclinic orbits.

For the scalar O.D.E. \eqref{efk}, the uniqueness (up to translations) of the minimal heteroclinic is a very difficult open problem. On the other hand, in the vector case ($m\geq 2$) explicit examples of potentials having at least two minimal heteroclinics can be given. More precisely, Lemma \ref{lem2} provides the existence of potentials $W$ for which the set $F$ of minimal heteroclinics of \eqref{ode0} satisfies the separation condition
\begin{equation}\label{parti}
F=F^-\cup F^+, \text{ with }  F^-\neq \varnothing, \  F^+\neq \varnothing, \text{ and } d_{\mathrm{min}}:=d(F^-,F^+)>0,
\end{equation}
where $d$ stands for the distance in $L^2(\R;\R^m)$, and $d(F^-,F^+):=\inf\{ \|e^--e^+\|_{L^2(\R;\R^m)}: e^-\in F^-, e^+\in F^+\}$. 
Under this structural assumption, we are going to construct heteroclinic double layers for \eqref{system}, that is, solutions $u(t,x)$ such that
\begin{subequations}\label{layer2}
\begin{equation}\label{lay1}
 \lim_{t\to\pm\infty} d(u(t,\cdot),F^\pm)=0,
\end{equation}
\begin{equation}\label{lay2}
\forall t\in\R: \ \lim_{x\to\pm\infty}u(t,x)=a^\pm.
\end{equation}
\end{subequations}
The existence of double layered solutions for the system $\Delta u-\nabla W(u)=0$, goes back to the work of Alama, Bronsard and Gui \cite{abg}. Subsequently, Schatzman \cite{scha} managed to remove the symmetry assumption on $W$ considered in \cite{abg}. We also mention the work of Alessio \cite{alessio}, where the separation condition \eqref{parti} has first been introduced, and the new developments on these results presented in \cite{fusco,monteil}.

Our construction of heteroclinic double layers for \eqref{system} is inspired in the approach from
Functional Analysis used in the classical theory of evolution equations. This method has recently been applied in the elliptic context (cf. \cite{ps2}) to give an alternative proof of Schatzman's result \cite{scha}. The idea is to view a solution $\R^2\ni (t,x)\mapsto u(t,x)$ of a P.D.E. as a map $t\mapsto [U(t):x\mapsto [U(t)](x):=u(t,x)]$ taking its values in an appropriate space of functions, and reduce the initial P.D.E. to an O.D.E. problem for $U$. Indeed, the uniform in $t$ boundary conditions \eqref{lay2} suggest to set
a map
\begin{equation*}
\ee_0(x)=
\begin{cases}
a^-,&\text{ for } x\leq -1,\\
\frac{a^++a^-}{2}+(a^+ -a^-)\frac{3x-x^3}{4},&\text{ for }-1 \leq x\leq 1,\\
a^+,&\text{ for } x\geq 1.
\end{cases}
\end{equation*}
and work in the affine subspace $\mathcal H:=\ee_0+L^2(\R;\R^m)=\{u=\ee_0+h: h\in  L^2(\R;\R^m)\}$ which has the structure of a Hilbert space with the inner product
\begin{equation*}\label{innerp}
\langle u,v \rangle_{\mathcal H}=\langle (u-\ee_0),(v-\ee_0)\rangle_{L^2(\R;\R^m)}, \ \forall u,v\in \mathcal H.
\end{equation*}
We also denote by $\|\cdot\|_{\mathcal H}$ the norm in $\mathcal H$, and by $d(u,v):=\|u-v\|_{L^2(\R;\R^m)}$ the corresponding distance.
In view of \eqref{expest}, it is clear that $e\in \mathcal{ H}$, and $e',e''\in L^2(\R;\R^m)$, for every minimal heteroclinic $e\in F$.

Next, we reduce system \eqref{system} together with the boundary conditions \eqref{layer2}, to a variational problem for the orbit $U:\R\to \mathcal H$, $t\mapsto [U(t):x\mapsto [U(t)](x):=u(t,x)]$. We shall proceed in several steps. 
The idea is to split between the variables $x$ and $t$, the terms appearing in the energy functional 
\begin{equation}\label{EE}
E_\Omega(u):=\int_\Omega\Big[\frac{1}{2}|\nabla^2 u|^2+\frac{\beta}{2}|\nabla u|^2+W(u)\Big], \ u \in H^2(\Omega;\R^m),\ \Omega\subset \R^2,
\end{equation}
associated to \eqref{system}. By gathering the derivatives of $u$ with respect to $x$, and the potential term, we first define in $\mathcal H$, the \emph{effective} potential $\mathcal W:\mathcal H \to [0,+\infty]$ by
\begin{equation}
\mathcal W(u)=
\begin{cases}
J_\R(u)-J_{\mathrm{min}}, &\text{ when the distributional derivatives }  u',\, u'' \in L^2(\R;\R^m),\\
+\infty,&\text{ otherwise,}
\end{cases}
\end{equation}
where $J_{\mathrm{min}}=\min_{v\in A} J_\R(v)$. 
Note that $\mathcal W\geq 0$, since $u' \in L^2(\R;\R^m)$ implies that $\lim_{x\to\pm\infty}u(x)=a^\pm$ i.e. $u \in A$, and thus $J_\R(u)\geq J_{\mathrm{min}}$. It is also obvious that $\mathcal W$ only vanishes on the set $F$ of minimal heteroclinics. 

Subsequently, we define the constrained class\footnote{The method of constrained minimization to construct minimal heteroclinics for the system $u''-\nabla W(u)=0$, goes back to \cite{alikakos1}. We refer to \cite{kreuter}, \cite{pap}, \cite{caz} and \cite{brezis2}, for the general theory of Sobolev spaces of vector-valued functions.}
\[\mathcal{A}=\Big\{V\in H_{\rm loc}^{2}(\R;\mathcal H):\left.  \begin{array}{l} V(t)\in \mathcal F^-,\text{ for }t\leq t_V^-,\\
V(t)\in \mathcal F^+,\text{ for }t\geq t_V^+,\end{array}\right.\text{ for some }t_V^-<t_V^+\Big\},\]
where
$\mathcal F^-=\{v+h: v \in \mathcal H, \, h \in L^2(\R;\R^m),\, d(v,F^-)\leq d(v,F^+), \|h\|_{L^2(\R;\R^m)}\leq d_{\mathrm{min}}/4\}$ (resp. $\mathcal F^+=\{v+h: v\in \mathcal H,\,  h \in L^2(\R;\R^m),\,d(v,F^+)\leq d(v,F^-), \|h\|_{L^2(\R;\R^m)}\leq d_{\mathrm{min}}/4\}$) are neighbourhoods of $F^-$ (resp. $F^+$) in $\mathcal H$, and the numbers $t_V^-<t_V^+$ depend on $V$.

Finally, we define the Action functional in $\mathcal H$ by 

\begin{equation}\label{action}
\mathcal J_{\R}(V):=\int_\R \Big[\frac{1}{2}\|V''(t)\|^2_{L^2(\R;\R^m)}+\frac{\beta}{2}\|V'(t)\|^2_{L^2(\R;\R^m)}+\sigma (V'(t))+\mathcal W(V(t))\Big] \dd t ,
\end{equation}
where we have set for $h\in L^2(\R;\R^m)$:
\begin{equation}
\sigma(h)=
\begin{cases}
\int_\R |h_x(x)|^2 \dd x, &\text{ when the distributional derivative }  h_x\in L^2(\R;\R^m),\\
+\infty,&\text{ otherwise,}
\end{cases}
\end{equation}
One can see that the definitions of $\mathcal W$ and $\mathcal J$ are relevant, since on a strip $[t_1,t_2]\times \R$, the energy functional $E$ is equal to $\mathcal J$ up to constant. More precisely, if $u \in C^2(\R^2;\R^m)$ is such that $u(x_1,x_2) \equiv a^+$, $\forall x_1\geq l>0$, and $u(x_1,x_2) \equiv a^-$, $\forall x_1\leq -l$, then setting $[U(t)](x):=u(t,x)$, one obtains $E_{[t_1,t_2]\times \R}(u)=\mathcal J_{[t_1,t_2]}(U)+(t_2-t_1)J_{\mathrm{min}}$.

In the proof of Theorem \ref{connh2} below, we show the existence of a minimizer $U$ of $\mathcal J$ in the constrained class $\mathcal A$. This result follows from an argument first introduced in \cite[Lemma 2.4.]{ps}, and from the nice properties of the effective potential $\mathcal W$ and the set $F$ established in section \ref{sec:sec2}. Let us just mention that $\mathcal W:\mathcal H\to[0,+\infty]$ is sequentially weakly lower semicontinuous (cf. Lemma \ref{lem1w}), and that the sets $F^\pm$ intersected with closed balls are compact in $\mathcal H$.  Next, from the orbit $U:\R\to\mathcal H$, we recover a solution $u$ of \eqref{system}. On the one hand, the constraint imposed in the class $\mathcal A$, forces $U$ to behave asymptotically as in \eqref{lay1}. On the other hand, the second boundary condition \eqref{lay2} follows from the definition of the space $\mathcal H$. In addition, since $U$ is a minimizer, the double layered solution $u$ obtained is \emph{minimal}, in the sense that
\begin{equation}\label{mineq}
E_{\mathrm{supp}\, \phi}(u)\leq E_{\mathrm{supp}\, \phi}(u+\phi), \ \forall \phi\in C^2_0(\R^2;\R^m).
\end{equation}
This notion of minimality is standard for many problems in which the energy of a localized solution is actually infinite due to non compactness of the domain. Thus, Theorem \ref{connh2} provides the existence of two-dimensional minimal solutions for \eqref{system}, whenever the potential $W$ satisfies the separation condition \eqref{parti}.

In contrast with second order phase transition models, the development of the theory of fourth order phase transition models in the P.D.E. context is very recent. The second order Allen-Cahn equation $\Delta u=u^3-u$, $u:\R^n\to\R$, has been the subject of a tremendous amount of publications in the past 30 years, certainly motivated from several challenging conjectures raised by De Giorgi (cf. \cite{dgiorgi1} and \cite{fav1}). As far as fourth order P.D.E. of phase transition type are concerned, only a few aspects of this theory have been investigated. Let us mention: the $\Gamma$-convergence results obtained in \cite{fonseca,hil}, the saddle solution constructed in \cite{bon16}, and the one-dimensional symmetry results established in \cite{bon15}, where an analog of the De Giorgi conjecture is stated, and a Gibbon's type conjecture is proved.

The aim of our present work is to provide for equation \eqref{system}, the first examples of two-dimensional \emph{minimal solutions}, since these solutions play a crucial role in phase transition models, and are closely related to the De Giorgi conjecture (cf. \cite{savin} in the case of second order phase transition models). After these explanations, we give the complete statement of Theorem \ref{connh2}:

\begin{theorem}\label{connh2}
Assume the potential $W$ satisfies \eqref{w13} and \eqref{parti}. Then, there exists a minimal solution $u\in C^4(\R^2;\R^m)$ of \eqref{system} such that
	such that
			\begin{subequations}\label{layer}
			\begin{equation}\label{lay1baa}
			\lim_{t\to\pm\infty}d(u(t,\cdot),F^\pm)=0, \text{ and } \lim_{t\to \infty} u_t(t,\cdot)=0 \text{ in } L^2(\R;\R^m),
  		\end{equation}
			\begin{equation}\label{lay2baa}
			\lim_{x\to\pm\infty}u(t,x)=a^\pm \text{ uniformly when $t$ remains bounded}.  
			\end{equation}
		\end{subequations}			
\end{theorem}

For the sake of simplicity we only focused in this paper on the extended Fisher-Kolmogorov equation, since it is a well-known fourth order phase transition model. However, the proof of Theorem \ref{connh2} can easily be adjusted to provide the existence of heteroclinic double layers for a larger class of P.D.E. than \eqref{system}. It trivially extends to operators such as $a_{11}u_{tttt}+a_{12}u_{ttxx}+a_{22}u_{xxxx}-b_1 u_{tt}-b_2u_{xx}$ (instead of $\Delta^2u-\beta \Delta u$), where $a_{ij}$, $b_i$ are positive constants.
We also point out that by dropping the term $\sigma (V'(t))$ appearing in the definition of $\mathcal J$, we still obtain a minimizer $U$ in the class $\mathcal A$, and thus a weak solution $u$ of
\begin{equation}\label{newpde}
\frac{\partial^4u(t,x)}{\partial t^4}+
\frac{\partial^4u(t,x)}{\partial x^4} -\beta \Delta u(t,x)+ \nabla W(u(t,x))=0, \ u:\R^2\to\R^m \ m\geq 1, \ \beta> 0, \ (t,x)\in\R^2,
\end{equation}
satisfying \eqref{lay1baa}. Finally, instead of the space $\mathcal H=\ee_0+L^2(\R;\R^m)$, other Hilbert spaces may be considered in the applications of Theorem \ref{connh2}. Indeed, since the properties of the effective potential $\mathcal W $ and the sets $F^\pm$ (established in section \ref{sec:sec2}) hold for the $H^2$ norm, we may construct an heteroclinic orbit $\tilde U$ connecting $F^-$ and $F^+$, either in $\ee_0+H^1(\R;\R^m)$ or $\ee_0+H^2(\R;\R^m)$. 
Then, we recover from each of these orbits, a weak solution of a sixth (resp. eighth) order P.D.E. satisfying \eqref{lay1baa}. We refer to \cite[section 5]{ps2} for a similar construction in the space $\ee_0+H^1(\R;\R^m)$, and for the adjustments to make in the proof of these results.

\section{Properties of the effective potential $\mathcal W$ and the sets $F^\pm$, $\mathcal F^\pm$}\label{sec:sec2}

Assuming that \eqref{w13} holds, we establish in Lemma \ref{lem1w} below some properties of the functions $\mathcal W$ and $\sigma$ defined in the previous section. We first recall two lemmas from \cite{ps}.

\begin{lemma}\label{lem1}\cite[Lemma 2.2.]{ps}
Let $A_b=\{ u\in A:\ J_{\R}(u)\leq J_0\}$, for some constant $J_0\geq J_{\mathrm{min}}$. Then, the maps $u\in \mathcal A_b$ as well as their first derivatives are uniformly bounded and equicontinuous.
\end{lemma}

\begin{lemma}\label{lem1new}\cite[Proof of Lemma 2.4.]{ps}
Given a sequence $\{u_k\}\in \mathcal H$, such that $\lim_{k\to\infty}\mathcal W(u_k)=0$, there exist a sequence $\{x_k\}\subset\R$, and a minimal heteroclinic $e\in F$, such that up to subsequence the maps $\bar u_k(x):=u_k(x-x_k)$ converge to $e$ in $C^1_{\mathrm{loc}}(\R;\R^m)$, as $k\to\infty$.
\end{lemma}

\begin{lemma}\label{lem1w}
\begin{itemize}
\item[(i)] The functions $\mathcal W:\mathcal H\to[0,+\infty]$ and $\sigma:L^2(\R;\R^m)\to[0,+\infty]$ are sequentially weakly lower semicontinuous.
\item[(ii)] Let $\{u_k\}\subset \mathcal H$ be such that $\lim_{k\to\infty}\mathcal W(u_k)=0$.
Then, there exist a sequence $\{x_k\}\subset\R$, and $e\in F$, such that (up to subsequence) the maps $\bar u_k(x):=u_k(x-x_k)$ satisfy $\lim_{k\to\infty}\|\bar u_k-e\|_{H^{2}(\R;\R^m)}=0$.
As a consequence, $d(u,F)\to 0$, as $\mathcal W(u)\to0$, and for every $c_1>0$, there exists $c_2>0$ such that $d(u,F)\geq c_1 \Rightarrow \mathcal W(u)\geq c_2$.
\end{itemize}
\end{lemma}

\begin{proof}
(i) Let $\{u_k\}\subset \mathcal H$ be such that $u_k \rightharpoonup u$ in $\mathcal H$ (i.e.
$u_k -u\rightharpoonup 0$ in $L^2(\R;\R^m)$), and let us assume that
$l=\liminf_{k\to \infty}\mathcal W(u_k)<\infty$ (since otherwise the statement is trivial).
By extracting a subsequence we may assume that $\lim_{k\to \infty}\mathcal W(u_k)=l$. Next, in view of Lemma \ref{lem1}, we can apply to the sequence $\{u_k\}$ the theorem of Ascoli, to deduce that $u_k \to  u$ in $C^1_{\mathrm{loc}}(\R;\R^m)$, as $k\to\infty$ (up to subsequence). On the other hand, since $\|u''_k\|_{L^2(\R;\R^m)}$ (resp. $\|u'_k\|_{L^2(\R;\R^m)}$) is bounded, we have that $u''_k \rightharpoonup v_2$, (resp. $u'_k \rightharpoonup v_1$) in $L^2(\R;\R^m)$ (up to subsequence). In addition, one can easily see that  $u\in H_{\rm loc}^{2}(\R;\R^m)$, and $u'=v_1$ as well as $u''=v_2$. Finally, by the weakly lower semicontinuity of the $L^2(\R;\R^m)$ norm we obtain $\|u''\|^2_{L^2(\R;\R^m)}\leq\liminf_{k\to\infty}\|u''_k\|^2_{L^2(\R;\R^m)}$ (resp.  $\|u'\|^2_{L^2(\R;\R^m)}\leq\liminf_{k\to\infty}\|u'_k\|^2_{L^2(\R;\R^m)}$, while by Fatou's Lemma we get $\int_\R W(u)\leq \liminf_{k\to\infty}\int_\R W(u_k)$. Gathering the previous results, we conclude that $\mathcal W(u)\leq l$ i.e.
$\mathcal W(u)\leq \liminf_{k\to \infty}\mathcal W(u_k)$. To show the sequentially weakly lower semicontinuity of $\sigma$, we proceed in a similar way. Let $\{h_k\}\subset L^2(\R;\R^m)$ be such that $h_k \rightharpoonup h$ in $L^2(\R;\R^m)$, and $\lim_{k\to \infty}\sigma(h_k)=l<\infty$. Since $\|h'_k\|_{L^2(\R;\R^m)}$ is bounded, we deduce that (up to subsequence) $h'_k \rightharpoonup g$ in $L^2(\R;\R^m)$ for some $g\in L^2(\R;\R^m)$, such that $\|g\|^2_{L^2(\R;\R^m)}\leq\liminf_{k\to\infty}\|h'_k\|^2_{L^2(\R;\R^m)}=l$. Clearly, we have
$h'=g$. Thus, we conclude that $\sigma(h)\leq l$ i.e. $\sigma(h)\leq \liminf_{k\to \infty}\sigma(h_k)$.

(ii) We first establish that given $u\in\mathcal H$ such that $u' ,u''\in L^2(\R;\R^m)$, and $e \in F$, we have
\begin{equation}\label{formula}
\mathcal W(u)=\int_\R \Big[\frac{1}{2}|u''-e''|^2+\frac{\beta}{2}|u'-e'|^2  +W(u)-W(e)-\nabla W(e)\cdot (u-e)\Big].
\end{equation}
In view of \eqref{expest}, it is clear that $e'$, $e''$, $e'''$, $e''''$ as well as $\nabla W(e)$ belong to $L^2(\R;\R^m)$. As a consequence, we can see that 
\begin{equation}\label{addition}
\int_\R [e''\cdot(u''-e'')+ \beta e'\cdot(u'-e')+\nabla W(e)\cdot (u-e)]=\int_\R [e''''-\beta e'' +\nabla W(e)]\cdot (u-e)=0.
\end{equation}
Finally, by substracting \eqref{addition} from $\mathcal W(u)=\int_\R \big[\frac{1}{2}|u''|^2-\frac{1}{2}|e''|^2+\frac{\beta}{2}|u'|^2-\frac{\beta}{2}|e'|^2 +W(u)-W(e)\big]$, formula \eqref{formula} follows.

Now, we consider a sequence $\{u_k\}\subset \mathcal H$ such that $\lim_{k\to\infty}\mathcal W(u_k)=0$. According to Lemma \ref{lem1new}, there exist a sequence $\{x_k\}\subset\R$, and $e\in F$, such that (up to subsequence) the maps $\bar u_k(x):=u_k(x-x_k)$ satisfy
\begin{equation}\label{propaa}
 \lim_{k\to\infty}\bar u_k(x)= e(x), \text{ in $C^1_{\mathrm{loc}}(\R;\R^m)$}.
\end{equation}
Our claim is that
\begin{equation}\label{claime}
\lim_{k\to\infty}\|\bar u_k-e\|_{H^{2}(\R;\R^m)}=0.
\end{equation}
According to hypothesis \eqref{w2} we have
\begin{subequations}
\begin{equation}\label{convex1}
	W(u)\geq\frac{c}{2}|u-a^\pm|^2 , \forall u: |u-a^\pm|\leq r,
	\end{equation}
\begin{equation}\label{convex2}
	W(v)-W(u)-\nabla W(u)\cdot (v-u)\geq \frac{c}{2}|v-u|^2, \forall u,v: |u-a^\pm|\leq r, \ |v-a^\pm|\leq r.
	\end{equation}
\end{subequations}
Let $\mu>0$ be such that
\begin{equation}\label{convex3}
	W(u)\leq\frac{\mu}{2}|u-a^\pm|^2 , \forall u\in\R^m: |u-a^\pm|\leq r,
	\end{equation}
and let $\epsilon\in (0,r/2)$. Given $v\in B:=\{v\in H^{2}([\lambda^-,\lambda^+];\R^m), |v(\lambda^\pm)-a^\pm|\leq \epsilon/8, |v'(\lambda^\pm)|\leq\epsilon/4\}$, we set $\phi(\lambda^\pm):=v(\lambda^\pm)-\frac{1}{2}v'(\lambda^\pm)$,
and define the comparison map
\begin{equation}\label{comparison3}
z(x)=\begin{cases}
v(\lambda^-)+\big((x-\lambda^-)+\frac{(x-\lambda^-)^2}{2}\big)v'(\lambda^-) &\text{ for } \lambda^--1\leq x \leq \lambda^-,\\
\phi(\lambda^-)+(2(x-\lambda^-+1)^2-(x-\lambda^-+1)^4)(a^--\phi(\lambda^-)) &\text{ for } \lambda^--2\leq x \leq \lambda^--1,\\
a^- &\text{ for }  x \leq \lambda^--2.
\end{cases}\end{equation}
An easy computation shows that 
\begin{itemize}
\item $z(\lambda^-)=v(\lambda^-)$, $z'(\lambda^-)=v'(\lambda^-)$, $z\in H^{2}_{\mathrm{loc}}((-\infty,\lambda^-];\R^m)$,
\item $\forall x\leq \lambda^-$: $|z(x)-a^-|\leq \epsilon/4$, $|z'(x)|\leq \epsilon$, $|z''(x)|\leq 2\epsilon$,
\item $J_{(-\infty,\lambda^-]}(z)\leq (4+\beta+\mu)\epsilon^2$.
\end{itemize} 
Clearly, by reproducing the same argument in the interval $[\lambda^+,\infty)$, we can find a comparison map $z\in H^{2}_{\mathrm{loc}}([\lambda^+,\infty);\R^m)$ such that $z(\lambda^+)=v(\lambda^+)$, $z'(\lambda^+)=v'(\lambda^+)$, and
$J_{[\lambda^+,\infty)}(z)\leq (4+\beta+\mu)\epsilon^2$. As a consequence, we obtain 
\begin{equation}\label{bound4}
\inf\{J_{[\lambda^-,\lambda^+]}(v): v\in B\}\geq J_{\mathrm{min}}-2(4+\beta+\mu)\epsilon^2,
\end{equation}
since otherwise we can construct a map in $A$ whose action is less than $J_{\mathrm{min}}$. On the other hand we have
\begin{equation}\label{bound5}
	\inf\{J_{[x_1,x_2]}(v): v\in H^{2}([x_1,x_2];\R^m), |v(x_1)-a^\pm|\leq \epsilon, |v(x_2)-a^\pm|= r\}\geq \sqrt{\beta c}(r/2)^2.
	\end{equation}
Indeed, for such a map $v$, there exists an interval $[\tilde x_1,\tilde x_2]\subset [x_1,x_2]$, such that $|v(\tilde x_1)-a^\pm|=r/2$, $|v(\tilde x_2)-a^\pm|=r$, and $|v(x)|\in [r/2,r]$, $\forall x \in[\tilde x_1,\tilde x_2]$, thus we can check that 
$$J_{[x_1,x_2]}(v)\geq J_{[\tilde x_1,\tilde x_2]}(v)\geq \int_{\tilde x_1}^{\tilde x_2} \sqrt{2\beta W(v)}|v'|\geq \sqrt{\beta c}(r/2)^2.$$
In the sequel, we fix an $\epsilon\in(0,r/2)$ such that $[2(4+\beta+\mu)+1]\epsilon^2<\sqrt{\beta c}(r/2)^2$, and choose an interval $[\lambda^-,\lambda^+]$ such that $|e(x)-a^-|\leq \epsilon/16$, $\forall x\leq \lambda^-$, $|e(x)-a^+|\leq \epsilon/16$, $\forall x\geq \lambda^+$, and  $|e'(x)|\leq \epsilon/8$, $\forall x\in \R\setminus (\lambda^-,\lambda^+)$. According to \eqref{propaa}, we have for $k\geq N$ large enough:
\begin{subequations}
\begin{equation}\label{bound6}
	|\bar u_k(\lambda^\pm)-a^\pm|< \epsilon/8, \ |\bar u'_k(\lambda^\pm)|< \epsilon/4,
	\end{equation}
\begin{equation}\label{bound7}
	\Big| \int_{[\lambda^-,\lambda^+]} ( W(\bar u_k)-W(e)-\nabla W(e)\cdot (\bar u_k-e))\Big|< \epsilon^2,
	\end{equation}
\begin{equation}\label{bound7b}
		\|\bar u_k-e\|_{L^2( [\lambda^-,\lambda^+];\R^m)}<\epsilon, 
	\end{equation}
\begin{equation}\label{bound8}
	\mathcal W(\bar u_k)<\epsilon^2.
	\end{equation}
\end{subequations}
Then, combining \eqref{bound4} with \eqref{bound8}, one can see that
\begin{equation}\label{bound9}
	J_{\R\setminus [\lambda^-,\lambda^+]}(\bar u_k)< [2(4+\beta+\mu)+1]\epsilon^2< \sqrt{\beta c}(r/2)^2.
	\end{equation}
Therefore, in view of \eqref{bound5} and \eqref{bound6}, it follows that $|\bar u_k(x)-a^-|\leq r$, $\forall x\leq \lambda^-$ (resp. $|\bar u_k(x)-a^+|\leq r$, $\forall x\geq \lambda^+$). Furthermore, as a consequence of \eqref{convex2} we get
\begin{equation}\label{bound10}
	\int_{\R\setminus [\lambda^-,\lambda^+]}	(W(\bar u_k)-W(e)-\nabla W(e)\cdot (\bar u_k-e))\geq \frac{c}{2} \|\bar u_k-e\|_{L^2( \R\setminus[\lambda^-,\lambda^+];\R^m)}^2.
	\end{equation}
To conclude, we apply formula \eqref{formula} to $\bar u_k$, and combine \eqref{bound8} with \eqref{bound7} and \eqref{bound10}, to obtain
\begin{equation}\label{bound11}
	\|\bar u_k-e\|_{L^2( \R\setminus [\lambda^-,\lambda^+];\R^m)}< \frac{2\epsilon}{\sqrt{c}}, \ \|\bar u'_k-e'\|_{L^2( \R;\R^m)}<\frac{2\epsilon}{\sqrt{\beta}}, \ \text{  and  } \|\bar u''_k-e''\|_{L^2( \R;\R^m)}<2\epsilon.
	\end{equation}
Finally, in view of \eqref{bound7b}, we have $\|\bar u_k-e\|_{L^2( \R;\R^m)}<\big(1+ \frac{2}{\sqrt{c}})\epsilon$. This establishes our claim \eqref{claime}, from which the statement (ii) of Lemma \ref{lem1w} is straightforward.
\end{proof}

From the arguments in the proof of Lemma \ref{lem1w}, we deduce some useful properties of the sets $F^\pm$ and $\mathcal F^\pm$ (defined in the previous section).
\begin{lemma}\label{propcon}
\begin{itemize}
\item[(i)] Let $\{e_k\}\subset F$ be bounded in $\mathcal H$, then
there exists $e\in F$, such that up to subsequence $\lim_{k\to\infty}\|e_k-e\|_{H^{2}(\R;\R^m)}=0$.
\item[(ii)] There exists a constant $\gamma>0$, such that for every $e\in F$, we can find $T \in \R$ such that setting $e^T (x)=e(x-T)$, we have $\|e^T-\ee_0\|_{H^2(\R;\R^m)}\leq \gamma$.
\item[(iii)] For every $v\in \mathcal H$, there exist $e^\pm\in F^\pm$ such that $d(v,F^\pm)=\|v-e^\pm\|_{L^2(\R;\R^m)}$. In particular, the functions $\mathcal H \ni v\mapsto d(v,F^\pm)$ are continuous.
\item[(iv)] The sets $\mathcal F^\pm$ are sequentially weakly closed in $\mathcal H$, and strongly closed in $\mathcal H$. Furthermore, we have 
$$\{v\in \mathcal H: d(v,F^-)\leq d_{\mathrm{min}}/2\}\subset \mathcal F^-, \ \{v\in \mathcal H: d(v,F^+) \leq d_{\mathrm{min}}/2\}\subset \mathcal F^+,$$ 
$$\mathcal F^-\cap \{v\in \mathcal H: d(v,F^+)<d_{\mathrm{min}}/4\}=\varnothing, \ \mathcal F^+\cap \{v\in \mathcal H: d(v,F^-)<d_{\mathrm{min}}/4\}=\varnothing.$$
$$\{v+h: v \in \mathcal H, \, h\in L^2(\R;\R^m),\, d(v,F^-)=d(v,F^+), \|h\|_{L^2(\R;\R^m)}\leq d_{\mathrm{min}}/4\}\subset\mathcal F^-\cap\mathcal F^+,$$
$$\partial \mathcal F^-\subset \{v\in \mathcal H: d(v,F^+) <d(v,F^-)\}\subset\mathcal F^+, \ \partial \mathcal F^+\subset \{v\in \mathcal H: d(v,F^-) <d(v,F^+)\}\subset\mathcal F^-.$$
\end{itemize}
\end{lemma}
\begin{proof}
(i) Since $\{e_k\}\subset F$ is bounded in $\mathcal H$, we have up to subsequence $e_k \rightharpoonup e$ in $\mathcal H$, as $k\to\infty$, for some $e\in\mathcal H$.
Proceeding as in the proof of Lemma \ref{lem1w} (i), we first obtain that (up to subsequence) $e_k \to  e$ in $C^1_{\mathrm{loc}}(\R;\R^m)$, as $k\to\infty$, with $e\in F$. Next, we reproduce the arguments after \eqref{claime}, with $e_k$ instead of $\bar u_k$.

(ii) Assume by contradiction the existence of a sequence $\N \ni k \mapsto e_k\in F$, such that $\|e^T_k-\ee_0\|_{H^2(\R;\R^m)}\geq k$, $\forall T\in \R$. Then, by Lemma \ref{lem1w} (ii), there exists a sequence $\{x_k\}\subset\R$, and $e\in F$, such that (up to subsequence) the maps $e_k^{x_k}$ satisfy $\lim_{k\to\infty}\| e_k^{x_k}-e\|_{H^2(\R;\R^m)}=0$. Clearly, this is a contradiction.

(iii) Let $\{e_k^\pm\}\subset F^\pm$ be sequences such that $\|v-e_k^\pm\|_{L^2(\R;\R^m)}\leq d(v,F^\pm)+\frac{1}{k}$, $\forall k$. Then, in view of (i) we have (up to subsequence) $e_k^\pm\to e^\pm$ in $\mathcal H$, as $k\to\infty$, with $e^\pm\in F^\pm$. As a consequence
$d(v,F^\pm)=\|v-e^\pm\|_{L^2(\R;\R^m)}$. 

(iv) We are going to check that $\mathcal F^-$ is sequentially weakly closed (the proof is similar for $\mathcal F^+$). Let $\{u_k\}\subset \mathcal F^-$ be a sequence such that $u_k \rightharpoonup u$ in $\mathcal H$. We write $u_k=v_k+h_k$ with $d(v_k,F^-)\leq d(v_k,F^+)$, and $\|h_k\|_{L^2(\R;\R^m)}\leq d_{\mathrm{min}}/4$. Up to subsequence, we have $h_k \rightharpoonup h$ in $L^2(\R;\R^m)$, with $\|h\|_{L^2(\R;\R^m)}\leq d_{\mathrm{min}}/4$. Thus, it also holds that $v_k \rightharpoonup v:=u-h$ in $\mathcal H$. 
Now, let $\{e_k^\pm\}\subset F^\pm$ be two sequences such that $d(v_k,e_k^\pm)=d(v_k,F^\pm)$. Since the sequences $\{e_k^\pm\}$ are bounded, it follows from (i), that (up to subsequences) $e_k^\pm\to e^\pm$ holds in $\mathcal H$, for some $e^\pm \in F^\pm$. Our claim is that $d(v,e^\pm)=d(v,F^\pm)$. Indeed, given $f^\pm\in F^\pm$, we have 
$$\|v_k-f^\pm\|^2_{\mathcal H}\geq \|v_k-e_k^\pm\|^2_{\mathcal H}\Leftrightarrow \|f^\pm\|^2_{\mathcal H}-2\langle v_k,f^\pm\rangle_{\mathcal H} \geq  \|e_k^\pm\|^2_{\mathcal H}-2\langle v_k,e_k^\pm\rangle_{\mathcal H},$$
and as $k\to\infty$, we get
$$ \|f^\pm\|^2_{\mathcal H}-2\langle v,f^\pm\rangle_{\mathcal H} \geq  \|e^\pm\|^2_{\mathcal H}-2\langle v,e^\pm\rangle_{\mathcal H}\Leftrightarrow\|v-f^\pm\|^2_{\mathcal H}\geq \|v-e^\pm\|^2_{\mathcal H},$$
which proves our claim. To show that $d(v,F^-)\leq d(v, F^+)$, we proceed as previously. By assumption, we have
$$\|v_k-e_k^-\|^2_{\mathcal H}\leq \|v_k-e_k^+\|^2_{\mathcal H}\Leftrightarrow  \|e_k^-\|^2_{\mathcal H}-2\langle v_k,e_k^-\rangle_{\mathcal H} \leq \|e_k^+\|^2_{\mathcal H}-2\langle v_k,e_k^+\rangle_{\mathcal H},$$
and as $k\to\infty$, we get
$$\|e^-\|^2_{\mathcal H}-2\langle v,e^-\rangle_{\mathcal H} \leq \|e^+\|^2_{\mathcal H}-2\langle v,e^+\rangle_{\mathcal H} \Leftrightarrow\|v-e^-\|^2_{\mathcal H}\leq \|v-e^+\|^2_{\mathcal H}.$$
This establishes that $\mathcal F^-$ is sequentially weakly closed, and thus also strongly closed. In view of the inequality $d(v,F^+)+d(v,F^-)\geq d_{\mathrm{min}}:=d(F^-,F^+)$, it is clear that $\{v\in \mathcal H: d(v,F^-)\leq d_{\mathrm{min}}/2\}\subset \mathcal F^-$ (resp. $\{v\in \mathcal H: d(v,F^+)\leq d_{\mathrm{min}}/2\}\subset \mathcal F^+$), and $\mathcal F^-\cap \{v\in \mathcal H: d(v,F^+)<d_{\mathrm{min}}/4\}=\varnothing$ (resp. $\mathcal F^+\cap \{v\in \mathcal H: d(v,F^-)<d_{\mathrm{min}}/4\}=\varnothing$). On the other hand, the inclusion $\{v+h:v \in \mathcal H, \, h\in L^2(\R;\R^m),\, d(v,F^-)=d(v,F^+), \|h\|_{L^2(\R;\R^m)}\leq d_{\mathrm{min}}/4\}\subset\mathcal F^-\cap\mathcal F^+$ follows immediately from the definition of $\mathcal F^\pm$. Finally, given $v \in\partial \mathcal F^-$ (resp. $v \in\partial \mathcal F^+$), we have $d(v,F^+) <d(v,F^-)$ (resp. $d(v,F^-) <d(v,F^+)$), since otherwise $v$ would be an interior point of $\mathcal F^-$ (resp. $\mathcal F^+$).
\end{proof}

Finally, in Lemma \ref{lem2} below, we give explicit examples of potentials for which the separation condition \eqref{parti} holds\footnote{A similar construction was performed in \cite[Remark 3.6.]{antonop} for the system $\Delta u-\nabla W(u)=0$.}.
\begin{lemma}\label{lem2}
Let $F\in C^2(\R^2;\R)$, $F(u)=\frac{(|u|^2-1)^2}{4}$ be the Ginzburg-Landau potential, let $\epsilon \in (0,1)$, and let $\phi \in C^\infty(\R,[0,\infty))$ be a function such that
\begin{equation*}\phi(t)=\begin{cases}
0 &\text {for } t\leq \sqrt{ 1-\epsilon}, \nonumber \\
1 &\text {for } \sqrt{1-\frac{\epsilon}{2}} \leq t.\nonumber
\end{cases} 
\end{equation*}
Then, the bistable potential $ W_\epsilon(u_1,u_2):=F(u)+u_2^2\phi(|u|^2)$ (with $u=(u_1,u_2)\in\R^2$), satisfies \eqref{parti} provided that $\epsilon\ll 1$.
\end{lemma}
\begin{proof}
Clearly, $W_\epsilon$ is a bistable potential vanishing at $a^\pm=(\pm1,0)$. One can easily check that it satisfies \eqref{w13}. 
Let $J_{\mathrm{min}}(\epsilon)$ be the action of a minimal heteroclinic for O.D.E. \eqref{ode0} with the potential $W_\epsilon$.
Our claim is that $J_{\mathrm{min}}(\epsilon)\to 0$, as $\epsilon\to 0$. To prove this, we construct a comparison map $v\in A$ as follows
\begin{equation*}
v(x)=\begin{cases}
a^-=(-1,0) &\text {for } x\leq -\frac{1}{\epsilon}-2, \\
\big(-1+\frac{\epsilon}{2}\big)+\frac{\epsilon}{4}(3(x+\frac{1}{\epsilon}+1)-(x+\frac{1}{\epsilon}+1)^3) &\text {for } -\frac{1}{\epsilon}-2 \leq x \leq -\frac{1}{\epsilon},\\
(1-\epsilon) e^{i\big(\frac{\pi}{4}(3(\epsilon x)-(\epsilon x)^3)-\frac{\pi}{2}\big)} &\text {for } -\frac{1}{\epsilon} \leq x \leq \frac{1}{\epsilon},\\
\big(1-\frac{\epsilon}{2}\big)+\frac{\epsilon}{4}(3(x-\frac{1}{\epsilon}-1)-(x-\frac{1}{\epsilon}-1)^3) &\text {for } \frac{1}{\epsilon} \leq x \leq \frac{1}{\epsilon}+2,\\
a^+=(1,0) &\text {for } x\geq \frac{1}{\epsilon}+2.
\end{cases} 
\end{equation*}
A long but otherwise not difficult computation shows that $J_\R(v)\leq 2\epsilon|2-\epsilon|^2+|1-\epsilon|^2(\beta \kappa_1\epsilon+\kappa_2\epsilon^3)$, for some constants $\kappa_1,\kappa_2>0$. This establishes that $J_{\mathrm{min}}(\epsilon)\to 0$, as $\epsilon\to 0$. On the other hand, one can see (cf. \eqref{bound5}) that
\begin{equation}\label{wizz}
	\inf\{J_{[x_1,x_2]}(v): v\in H^{2}([x_1,x_2];\R^m), |v(x_1)|\geq \sqrt{3}/2, |v(x_2)|\leq 1/2 \}\geq \frac{\sqrt{2\beta}(\sqrt{3}-1)}{16}
	\end{equation}
holds, provided that $\epsilon<1-\frac{\sqrt{3}}{2}$. As a consequence, if $\epsilon$ is small enough (such that $J_{\mathrm{min}}(\epsilon)< \frac{\sqrt{2\beta}(\sqrt{3}-1)}{16}$), then every minimal heteroclinic $e\in F$ takes its values into $\{u \in\R^2: |u|\geq 1/2\}$. We also notice that since the potential $W_\epsilon$ is invariant by the reflection with respect to the $u_1$ coordinate axis, $x\mapsto e(x):=(e_1(x),e_2(x))\in\R^2$ is a minimal heteroclinic iff $x\mapsto e(x):=(e_1(x),-e_2(x))\in\R^2$ is a minimal heteroclinic. Now, let $F^+$ (resp. $F^-$) be the set of minimal heteroclinics connecting $a^-$ to $a^+$ in the clockwise (resp. counterclockwise) direction. In view of the aforementioned symmetry property, it is clear that $F^-\neq \varnothing$ and $F^+\neq \varnothing$. Our claim is that $d(F^-,F^+)>0$. To check this, let $\mu:=\sup\{\|e '\|_{L^\infty(\R;\R^m)}\cdot \|e\|_{L^\infty(\R;\R^m)}: e\in F\}<\infty$ (cf. Lemma \ref{lem1}), and given $e^\pm\in F^\pm$, let $\psi(x):=|e^+(x)-e^-(x)|^2$. Since we have $|\psi'|\leq 8\mu$, and $\psi(x)\geq \psi (x_0)-8\mu |x-x_0|$, $\forall x,x_0\in\R$, the condition $\psi(x_0)\geq \frac{1}{4}$ for some $x_0\in\R$, implies that $\int_{\R} \psi\geq \frac{1}{128\mu}$. Finally, we notice that $\|e^+-e^-\|_{L^\infty(\R;\R^m)}^2\leq \frac{1}{4}$ does not hold, since otherwise the orbits of $e^+$ and $e^-$ would be homotopic in the set $\{u \in\R^2: |u|\geq 1/4\}$. This proves that $d(e^-,e^+)\geq \frac{1}{\sqrt{128\mu}}$.
\end{proof}

\section{Proof of Theorem \ref{connh2}}\label{sec:sec33}

\begin{proof}[Existence of the minimizer $U$]
We first establish that $\inf_{V\in\mathcal{A}}\mathcal J_{\R}<\infty$. Indeed, given $e^\pm_0\in F^\pm$, let
\begin{equation}\label{minfun}
\VV_0(t)=
\begin{cases}
e^-_0,&\text{ for } t\leq -1,\\
\frac{e^+_0+e^-_0}{2}+(e_0^+ -e_0^-)\frac{3t-t^3}{4},&\text{ for }-1 \leq t\leq 1,\\
e_0^+,&\text{ for } t\geq 1.
\end{cases}
\end{equation}
One can check that $\VV_0\in\mathcal A$, and $\mathcal J_0:=\mathcal J_{\R}(\VV_0)<\infty$, since $e^-_0$ and $e^+_0$ satisfy the exponential estimate \eqref{expest}. Setting $\mathcal{A}_b=\{V\in\mathcal A: \mathcal J_{\R}(V)\leq \mathcal J_0\}$, it is clear that $\inf_{V\in\mathcal{A}}\mathcal J_{\R}(V)=\inf_{V\in\mathcal{A}_b}\mathcal J_{\R}(V)<+\infty$. In the following lemma, we establish some properties of finite energy orbits $V\in\mathcal{A}_b$:

\begin{lemma}\label{lem1aa}
There exist a constant $M>0$ such that $\|V(t_2)-V(t_1)\|_{L^2(\R;\R^m)}\leq M |t_2-t_1|^{1/2}$, $\forall t_1, t_2\in\R$, $\forall V\in\mathcal{A}_b$, and a constant $M'>0$ such that $\|V'(t)\|_{L^2(\R;\R^m)}\leq M'$, $\forall t\in\R$, $\forall V\in\mathcal{A}_b$. Moreover every map $V\in\mathcal A_b$ satisfies $\lim_{t\to\pm\infty}d(V(t),F^\pm)=0$, and $\lim_{t\to\pm\infty}\|V'(t)\|_{L^2(\R;\R^m)}=0$.
\end{lemma}
\begin{proof}
It is clear that for every $t_1<t_2$, and every $V\in\mathcal{A}_b$, we have
\begin{equation*}
\|V(t_2)-V(t_1)\|_{L^2(\R;\R^m)}\leq \int_{t_1}^{t_2}\|V'(s)\|_{L^2(\R;\R^m)}\dd t\leq \Big|\int_{t_1}^{t_2} \|V'(s)\|^2_{L^2(\R;\R^m)}\dd t\Big|^{1/2}|t_2-t_1|^{1/2}\leq M |t_2-t_1|^{1/2},
\end{equation*}
with $M=\sqrt{2\mathcal J_0/\beta}$. To establish that $\lim_{t\to\pm\infty}d(V(t),F)=0$, assume by contradiction the existence of a sequence $t_k$ such that $\lim_{k\to\infty}| t_k|=\infty$, and $d(V(t_k),F)\geq 2\epsilon$, for some $\epsilon>0$. According to what precedes we have $d(V(t),F)\geq \epsilon$, $\forall t\in [t_k-\eta,t_k+\eta]$, with $\eta:= (\epsilon/M)^2$. Thus, Lemma \ref{lem1w} (ii) implies that $\mathcal W(V(t))\geq w_\epsilon>0$ holds for every $t\in [t_k-\eta,t_k+\eta]$, with $w_\epsilon:=\inf\{\mathcal W(u): d(u,F)\geq\epsilon\}>0$. In addition, since the intervals $[t_k-\eta,t_k+\eta]$ may be assumed to be disjoint, we obtain $\mathcal J_\R(V)=\infty$, which is a contradiction. Now, it remains to show that $\lim_{t\to\pm\infty}d(V(t),F^\pm)=0$. This property follows from the fact that in a neighbourhood of $-\infty$ (resp. $+\infty$), we have $V(t) \in\mathcal F^-\subset \{v\in \mathcal H: d(v,F^+)\geq d_{\mathrm{min}}/4\}$ (resp. $V(t)\in\mathcal F^+\subset\{v\in \mathcal H: d(v,F^-)\geq d_{\mathrm{min}}/4\}$), in view of Lemma \ref{propcon} (iv). Finally, to prove that $\lim_{t\to\pm\infty}\|V'(t)\|_{L^2(\R;\R^m)}=0$, and $\sup_{t\in\R}\|V'(t)\|_{L^2(\R;\R^m)}\leq M'$, $\forall V\in\mathcal{A}_b$, we notice that $V'$ belongs to $H^1(\R;L^2(\R;\R^m))\hookrightarrow L^\infty(\R;L^2(\R;\R^m))$, and $\|V'\|_{H^1(\R;L^2(\R;\R^m))}$ is uniformly bounded for $V\in\mathcal{A}_b$.
\end{proof}
%%%%%%%%%%%%%%%%%%%%%%%%%%%%%%
\begin{figure}[h]
\begin{center}
\includegraphics[scale=1]{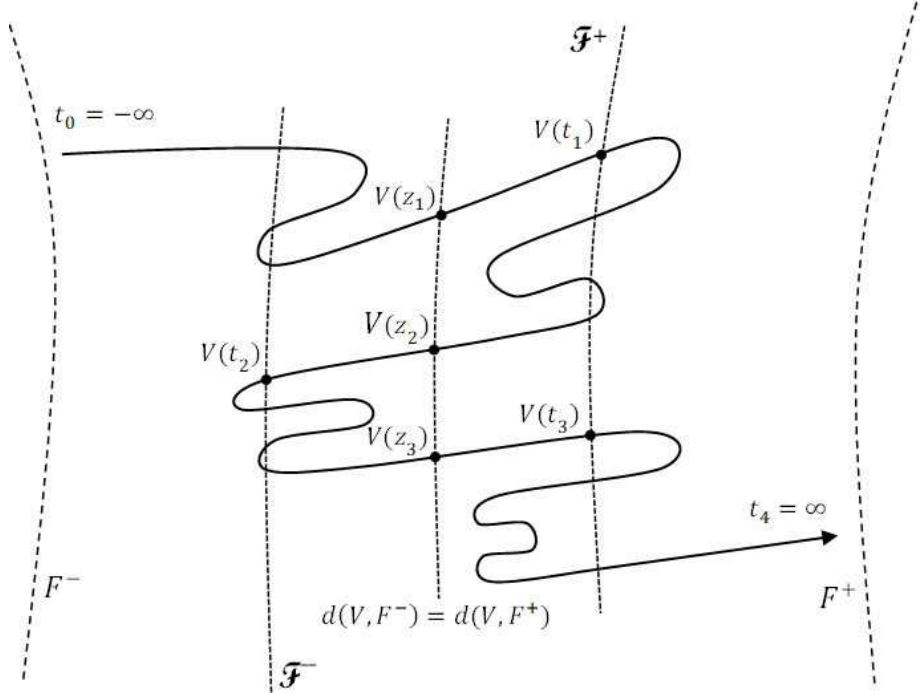}
\end{center}
\caption{The sequence $-\infty=t_0<z_1<t_1< z_2<t_2<\ldots<t_{2N}=\infty$, ($N=2$).}
\label{fig}
\end{figure}
%%%%%%%%%%%%%%%%%%%%%%%%%%
Now, let $\{V_k\}\subset\mathcal A_b$ be a minimizing sequence, i.e. $\lim_{k\to\infty}\mathcal J_\R(V_k)=\inf_{V\in\mathcal{A}_b}\mathcal J_{\R}(V)$.
For every $k$ we define the sequence $$-\infty<t_1(k)<t_2(k)<\ldots<t_{2N_k-1}(k)<t_{2N_k}(k)=\infty$$
by induction:
\begin{itemize}
\item $t_1(k)=\sup\{ t\in\R: \, V_k(s)\in\mathcal F^-, \forall s\leq t\}<\infty$ (note that $V_k(t_1(k))\in \partial\mathcal F^-\subset \mathcal F^+$, and $d(V_k(t_1(k)), F^+)<d(V_k(t_1(k)),F^-)$, in view of Lemma \ref{propcon} (iv)),
\item $t_{2i}(k)=\sup\{ t\geq t_{2i-1}(k): \, V_k(s)\in\mathcal F^+, \forall s\in [t_{2i-1}(k),t]\}\leq\infty$  (note that $t_{2i}(k)<\infty$ implies that $V_k(t_{2i}(k))\in \partial\mathcal F^+\subset \mathcal F^-$, and $d(V_k(t_{2i}(k)), F^-)<d(V_k(t_{2i}(k)),F^+)$, in view of Lemma \ref{propcon} (iv)),
\item $t_{2i+1}(k)=\sup\{ t\geq t_{2i}(k): \, V_k(s)\in\mathcal F^- , \forall s\in [t_{2i}(k),t]\}<\infty$, if $t_{2i}(k)<\infty$ (again we have $V_k(t_{2i+1}(k))\in \partial\mathcal F^-\subset \mathcal F^+$, and $d(V_k(t_{2i+1}(k)), F^+)<d(V_k(t_{2i+1}(k)),F^-)$, in view of Lemma \ref{propcon} (iv)),
\end{itemize}
where $i=1,\ldots,N_k$.
In addition, we set 
\begin{itemize}
\item $z_{2i-1}(k)=\sup\{ t\leq t_{2i-1}(k): \,  d(V_k(t), \mathcal F^-)=d(V_k(t), \mathcal F^+)\}$,
\item $z_{2i}(k)=\sup\{ t\leq t_{2i}(k): \,   d(V_k(t), \mathcal F^-)=d(V_k(t), \mathcal F^+)\}$, if $t_{2i}(k)<\infty$,
\end{itemize}
and define $b_j(k)=\sup \{t\geq z_j(k): \|V_k(s)-V_k(z_j(k))\|_{L^2(\R;\R^m)}<  d_{\mathrm{min}}/4, \forall s \in [z_j(k),t]\}$, for $j=1,\ldots,2N_k-1$. Since the set $\{v+h: v \in \mathcal H, \, h\in L^2(\R;\R^m), \,d(v,F^-)=d(v,F^+), \|h\|_{L^2(\R;\R^m)}< d_{\mathrm{min}}/4\}$ is included in the interior of $\mathcal F^-$ (resp. $\mathcal F^+$), it is clear that $b_j(k)\leq t_j(k)$. In addition, we have $\inf\{\mathcal W(V_k(t)): t\in(z_j(k),b_j(k)) \}\geq \mathcal W_0:=\inf \{\mathcal W(v): d(v,F)\geq d_{\mathrm{min}}/4\}>0$, in view of Lemma \ref{propcon} (iv) and Lemma \ref{lem1w} (ii). Thus, for every $k\geq 1$ and $j=1,\ldots, 2N_k-1$, we obtain
\begin{equation}\label{we1}
\mathcal J_{(z_j(k),b_j(k))}(V_k)\geq  \int_{z_j(k)}^{b_j(k)}\sqrt{2\beta \mathcal W(V_k(t))}\|V'_k(t)\|_{L^2(\R;\R^m)}\dd t\geq \sqrt{2\beta \mathcal W_0} (d_{\mathrm{min}}/4) .\nonumber
\end{equation}
This implies that $(2N_k-1)\sqrt{2\beta \mathcal W_0} (d_{\mathrm{min}}/4) \leq \mathcal J_0$, i.e.
the integers $N_k$ are uniformly bounded. By passing to a subsequence, we may assume that $N_k$ is a constant integer $N\geq 1$.

Our next claim (cf. \cite[Lemma 2.4.]{ps}) is that up to subsequence, there exist an integer $i_0$ ($1\leq i_0\leq N$) and an integer $j_0$ ($i_0\leq j_0\leq N$) such that
\begin{itemize}
\item[(a)] the sequence $t_{2j_0-1}(k)-t_{2i_0-1}(k)$ is bounded,
\item[(b)] $\lim_{k\to\infty}(t_{2i_0-1}(k)-t_{2i_0-2}(k))=\infty$,
\item[(c)] $\lim_{k\to\infty}(t_{2j_0}(k)-t_{2j_0-1}(k))=\infty$,
\end{itemize}
where for convenience we have set $t_0(k):=-\infty$. Indeed, we are going to prove by induction on $N\geq 1$, that given $2N+1$ sequences $-\infty\leq t_0(k)<t_1(k)<\ldots<t_{2N}(k)\leq\infty$, such that $\lim_{k\to\infty}(t_1(k)-t_0(k))=\infty$, and $\lim_{k\to\infty}(t_{2N}(k)-t_{2N-1}(k))=\infty$, then up to subsequence the properties (a), (b), and (c) above hold, for two fixed indices $1\leq i_0\leq j_0\leq N$.
When $N=1$, the assumption holds by taking $i_0=j_0=1$. Assume now that $N>1$, and let $l \geq 1$ be the largest integer such that
the sequence $t_{l}(k)-t_{1}(k)$ is bounded. Note that $l<2N$. If $l$ is odd, we are done, since the sequence $t_{l+1}(k)-t_{l}(k)$ is unbounded, and thus we can extract a subsequence $\{n_k\}$ such that $\lim_{k\to\infty}(t_{l+1}(n_k)-t_{l}(n_k))=\infty$. Otherwise $l=2m$ (with $1\leq m<N$), and the sequence $t_{2m+1}(k)-t_{2m}(k)$ is unbounded. We extract a subsequence $\{n_k\}$ such that $\lim_{k\to\infty}(t_{2m+1}(n_k)-t_{2m}(n_k))=\infty$.
Then, we apply the inductive statement with $N'=N-m$, to the $2N'+1$ sequences $t_{2m}(n_k)<t_{2m+1}(n_k)<\ldots<t_{2N}(n_k)$.

To show the existence of the minimizer $U$, we shall consider appropriate translations of the sequence $v_k(t,x):=[V_k(t)](x)$ ($\R\ni t\mapsto V_k(t)\in \mathcal H$), with respect to both variables $x$ and $t$. Then, we shall establish the convergence of the translated maps to the minimizer $U$. Given $T\in\R$, and $V\in \mathcal H=\ee_0+L^2(\R;\R^m)$, we denote by $L^T(V)$ the map of $\mathcal H$ defined by $\R\ni x\mapsto V(x-T)\in \R^m$. It is obvious that $\mathcal W(L^T(V))=\mathcal W(V)$.
Similarly, if $t\mapsto V(t)$ belongs to $H^{2}_{\mathrm{loc}}(\R;\mathcal H)$, we obtain that $t\mapsto L^T(V(t))$ also belongs to $H^{2}_{\mathrm{loc}}(\R;\mathcal H)$, with 
\begin{itemize}
\item $(L^T V)'=L^T(V')$, $\sigma ((L^T V)')=\sigma(V')$, and $\|(L^T V)'(t)\|_{L^2(\R;\R^m)}=\|V'(t)\|_{L^2(\R;\R^m)}$,
\item $(L^T V)''=L^T(V'')$, and $\|(L^T V)''(t)\|_{L^2(\R;\R^m)}=\|V''(t)\|_{L^2(\R;\R^m)}$.
\end{itemize}

At this stage, we infer that the sequence $d(V_k(t_{2i_0-1}(k)), F^-)$ is bounded. Indeed, when $k$ is large enough, we have $V_k([t_{2i_0-1}(k)-\frac{2\mathcal J_0}{\mathcal W_0}, t_{2i_0-1}(k)])\subset \mathcal F^-$ (where $\mathcal W_0:=\inf \{\mathcal W(v): d(v,F)\geq d_{\mathrm{min}}/4\}>0$). Thus, there exists $s_k\in [t_{2i_0-1}(k)-\frac{2\mathcal J_0}{\mathcal W_0}, t_{2i_0-1}(k)]$ such that $d(V_k(s_k),F^-)<d_{\mathrm{min}}/4$, 
since otherwise we would obtain $\int_{t_{2i_0-1}(k)-\frac{2\mathcal J_0}{\mathcal W_0}}^{ t_{2i_0-1}(k)}\mathcal W(V_k(t))\dd t\geq 2\mathcal J_0$, which is a contradiction. This proves that 
$d(V_k(t_{2i_0-1}(k)), F^-)\leq d_{\mathrm{min}}/4+d(V_k(s_k),V_k(t_{2i_0-1}(k)))\leq \eta:=d_{\mathrm{min}}/4+M\sqrt{\frac{2\mathcal J_0}{\mathcal W_0}}$.
We also claim that the sets $\mathcal F^\pm$ are invariant by the translations $L^T$. To check this, let us pick $u\in \mathcal F^-$ (the proof is similar for $\mathcal F^+$). By definition, $u=v+h$, with $v\in \mathcal H$, $h\in L^2(\R;\R^m)$, $d(v,F^-)\leq d(v,F^+)$, and $\|h\|_{L^2(\R;\R^m)}\leq d_{\mathrm{min}}/4$. If $d(v,F^\pm)=d(v,e^\pm)$, for some $e^\pm\in F^\pm$, one can see that $d(L^T v,F^\pm)=d(L^T v,L^T e^\pm)=d(v,e^\pm)$. Therefore, we have $L^T u=L^T v+L^T h$, with $d(L^T v,F^-)\leq d(L^T v,F^+)$, and $\|L^T h\|_{L^2(\R;\R^m)}=\| h\|_{L^2(\R;\R^m)}\leq d_{\mathrm{min}}/4$, i.e. $L^Tu\in \mathcal F^-$.

Next, in view of Lemma \ref{propcon} (ii), for every $k$, we can find $T_k\in \R$ and $e_k \in F^-$ such that
$\|e_k\|_{\mathcal H}\leq \gamma$ and $\||L^{T_k}V_k(t_{2i_0-1}(k))-e_k\|_{\mathcal H}\leq \eta$. We set $\bar V_k(t):=L^{T_k}(V_k(t+t_{2i_0-1}(k)))$. Clearly, $\bar V_k \in H^{2}_{\mathrm{loc}}(\R;\mathcal H)$ satisfies $\mathcal J_\R(\bar V_k)=\mathcal J_\R(V_k)$, as well as
\begin{equation}\label{solinas2}
\bar V_k(t)\in \mathcal F^-, \forall t\in [t_{2i_0-2}(k)- t_{2i_0-1}(k),0], \text{ and } \bar V_k(t) \in\mathcal F^+ , \forall t\in [t_{2j_0-1}(k)- t_{2i_0-1}(k),t_{2j_0}(k)- t_{2i_0-1}(k)].
\end{equation}
On the one hand, since $\|\bar V_k(0)\|_{\mathcal H}\leq \eta+\gamma$ holds for every $ k$, 
we have that (up to subsequence) $\bar V_k(0) \rightharpoonup  u_0$ in $\mathcal H$, as $k\to\infty$, for some $u_0\in\mathcal H$.
On the other hand, since $\bar V'_k$ as well as $\bar V''_k$ are uniformly bounded in $L^2(\R;L^2(\R;\R^m))$, it follows that up to subsequence
\begin{equation}\label{la1}
\bar V'_k\rightharpoonup U_1, \text{ and }\bar V''_k\rightharpoonup U_2 \text{ hold in } L^2(\R,L^2(\R;\R^m)),
\end{equation}
for some $U_1,U_2\in L^2(\R;L^2(\R;\R^m))$, such that $U_1'=U_2$, and
\begin{subequations}\label{lala2}
\begin{equation}
\int_{\R}\| U_1(t)\|^2_{L^2(\R;\R^m)}\dd t\leq\liminf_{k\to\infty} \int_{\R}\|\bar V'_k(t)\|^2_{L^2(\R;\R^m)}\dd t,
\end{equation}
\begin{equation}
\int_{\R}\| U_2(t)\|^2_{L^2(\R;\R^m)} \dd t\leq\liminf_{k\to\infty} \int_{\R}\|\bar V''_k(t)\|^2_{L^2(\R;\R^m)}\dd t.
\end{equation}
\end{subequations}
Finally, we write $\bar V_k(t)=\bar V_k(0)+\int_0^t\bar V'_k(s)\dd s$,
and claim that $U(t):=u_0+\int_0^tU_1(s)\dd s$ is a minimizer of $\mathcal J$ in $\mathcal A$. Indeed, it is clear that $U\in H^2_{\mathrm{loc}}(\R;\mathcal H)$, and since $\int_0^t\bar V'_k(s)\dd s\rightharpoonup  \int_0^t U_1(s)\dd s$ holds in $\mathcal H$ for every $t\in\R$, we also have $\bar V_k(t)\rightharpoonup U(t)$ for every $t\in\R$.  Similarly, since $\|V'_k(0)\|_{L^2(\R;\R^m)}$ is uniformly bounded (cf. Lemma \ref{lem2}), it follows that (up to subsequence) $\bar V'_k(0) \rightharpoonup  u_1$ in $L^2(\R;\R^m)$. Thus, for every $t\in\R$, we have $\bar V'_k(t)=\bar V'_k(0)+\int_0^t\bar V'_k(s)\dd s \rightharpoonup u_1+ \int_0^t U_2(s)\dd s=U_1(t)+h$, for some $h\in L^2(\R;\R^m)$, that we are going to determine. On the one hand, in view of the bound $\|V'_k(t)\|_{L^2(\R;\R^m)}\leq M'$, $\forall k$, $\forall t\in\R$, we obtain by dominated convergence that
$\lim_{k\to\infty}\int_0^t\langle \bar V'_k(s),h\rangle_{L^2(\R;\R^m)} \dd s=\int_0^t\langle  U_1(s),h\rangle_{L^2(\R;\R^m)} \dd s+t\|h\|^2_{L^2(\R;\R^m)}$. On the other hand, using the weak convergence $\bar V'_k\rightharpoonup U_1$ in $L^2(\R;L^2(\R;\R^m))$, we deduce that
$\lim_{k\to\infty}\int_0^t\langle \bar V'_k(s),h\rangle_{L^2(\R;\R^m)} \dd s=\int_0^t\langle  U_1(s),h\rangle_{L^2(\R;\R^m)} \dd s$. Thus $h=0$, and we have established that $\bar V'_k(t) \rightharpoonup U_1(t)$ holds for every $t\in\R$. Now, the sequentially weakly lower semicontinuity of $\mathcal W$ and $\sigma$ (cf. Lemma \ref{lem1w} (i)), implies that $\liminf_{k\to\infty}[\sigma(\bar V'_k(t))+\mathcal W(\bar V_k(t))]\geq [\sigma( U_1(t))+\mathcal W(U(t))]$ for every $t\in\R$, thus by Fatou's Lemma we obtain
\begin{equation}\label{la2}
\int_{\R}[\sigma( U_1(t))+\mathcal W(U(t))]\dd t\leq\liminf_{k\to\infty} \int_{\R}[\sigma(\bar V'_k(t))+\mathcal W(\bar V_k(t))]\dd t.
\end{equation}
Combining \eqref{la1}\ with \eqref{lala2} it is clear that $\mathcal J_\R(U)\leq\liminf_{k\to\infty}\mathcal J_{\R}(V_k)$. To conclude it remains to show that $U\in\mathcal A$. In view of the above property (b) it follows that $U(t)\in \mathcal F^-$, for every $t\leq 0$. Similarly, in view of (a) and (c), we have $U(t)\in\mathcal F^+$, for $t\geq T>0$ large enough. 
\end{proof}

\begin{proof}[Existence of the double layered solution] We first identify $U$ with a map $u \in H^2_{\mathrm{loc}}(\R^2;\R^m)$:
\begin{lemma}\label{sobolev}
Writing $U(t)=\ee_0+H(t)$, with $$H\in H^2_{\mathrm{loc}}(\R;L^2(\R;\R^m))\subset L^2_{\mathrm{loc}}(\R;L^2(\R;\R^m)),$$ and identifying $H$ with a $L^2_{\mathrm{loc}}(\R^2;\R^m)$ map via $h(t,x):=[H(t)](x)$, we have 
\begin{itemize}
\item $h \in H^2_{\mathrm{loc}}(\R^2;\R^m)$, $h_t ,h_{tt},h_{tx}\in L^2(\R^2;\R^m)$, and $h_x,h_{xx}\in L^2((t_1,t_2)\times \R;\R^m)$ for every interval $[t_1,t_2]\subset\R$.
\item Moreover, $\|h_x\|_{L^2((t_1,t_2)\times \R;\R^m)}^2+ \|h_{xx}\|_{L^2((t_1,t_2)\times \R;\R^m)}^2\leq C_0 (|t_2-t_1|)$, for a constant $C_0>0$ depending only on $|t_2-t_1|$.
\end{itemize}
\end{lemma}
\begin{proof}
We recall that given any interval $(t_1,t_2)$, we can identify $L^2((t_1,t_2)\times \R;\R^m)$ with $L^2((t_1,t_2);L^2(\R;\R^m))$ via the canonical isomorphism
\begin{align*}\label{isom}
L^2((t_1,t_2)\times \R;\R^m)&\simeq L^2((t_1,t_2);L^2(\R;\R^m))\\
f&\simeq [(t_1,t_2)\ni t \mapsto [F(t)]:x\mapsto f(t,x)], \ F(t)\in L^2(\R;\R^m).
\end{align*}
Let $g_1(t,x):=[U'(t)](x)$, $g_2(t,x):=[U''(t)](x)$, with $g_1,g_2\in L^2(\R^2;\R^m)$, and let us prove that $h_t=g_1$.
Given a function $\phi \in C^\infty_0(\R^2;\R^m)$, we also view it as a map $\Phi\in C^1(\R; L^2(\R;\R^m))$, $t\mapsto \Phi(t)$, by setting $[\Phi(t)](x):=\phi(t,x)$.
Assuming that $\supp\Phi\subset(t_1,t_2)$, we have $$\int_{\R^2} [h\cdot\phi_t+g_1\cdot\phi]=\int_{t_1}^{t_2} (\langle H(t),\Phi_t(t)\rangle_{L^2(\R;\R^m)}+\langle H_t(t),\Phi(t)\rangle_{L^2(\R;\R^m)})\dd t,$$
and clearly the second integral vanishes if $H\in C^1([t_1,t_2]; L^2(\R;\R^m))$. Since $H$ can be approximated in $H^1((t_1,t_2); L^2(\R;\R^m))$ by $C^1([t_1,t_2]; L^2(\R;\R^m))$ maps, we deduce that $\int_{\R^2} [h\cdot\phi_t+g_1\cdot\phi]=0$, i.e. $h_t=g_1$. Similarly, we can prove that $h_{tt}=\frac{\partial g_1}{\partial t}=g_2$. On the other hand, to establish that $h_{tx}\in L^2(\R^2;\R^m)$, we use difference quotients. Indeed, for a.e. $t\in\R$, and for every $\xi \in (-1,1)$, we have
\begin{equation}\label{soba2}
\int_\R\Big|\frac{g_1(t,x+\xi)-g_1(t,x)}{\xi}\Big|^2\dd x \leq \sigma(g_1(t,\cdot))=\sigma (U'(t))<\infty,
\end{equation}
thus after an integration, we obtain 
\begin{equation}\label{soba3}
\int_{\R^2}\Big|\frac{g_1(t,x+\xi)-g_1(t,x)}{\xi}\Big|^2\dd t \dd x \leq \int_{\R}\sigma(U'(t))\dd t<\infty\Rightarrow h_{tx}\in L^2(\R^2;\R^m).
\end{equation}
Finally, since $\int_\R\mathcal W(U(t))\dd t<\infty$ it follows that for a.e. $t\in\R$, we have $\mathcal W(U(t))<\infty$, and $U(t)\in \ee_0+H^2(\R;\R^m)$.
By using again difference quotients, we can see that
\begin{equation}\label{sob2}
\int_\R\Big|\frac{h(t,x+\xi)-h(t,x)}{\xi}\Big|^2\dd x \leq \int_\R |h_x|^2 \leq \frac{4(\mathcal W(U(t))+J_{\mathrm{min}})}{\beta} +2\|\ee'_0\|_{L^2(\R;\R^m)}^2,
\end{equation}
holds for a.e. $t\in\R$, and for $\xi\in(-1,1)$. As a consequence, the difference quotients $\frac{h(t,x+\xi)-h(t,x)}{\xi}$ (with $\xi\in(-1,1)$) are bounded in $L^2((t_1,t_2)\times\R;\R^m)$ for every interval $[t_1,t_2]\subset\R$, since an integration of \eqref{sob2} gives 
\begin{equation}\label{sob3}
\int_{t_1}^{t_2}\int_\R\Big|\frac{h(t,x+\xi)-h(t,x)}{\xi}\Big|^2\dd t\dd x \leq \frac{4\int_\R\mathcal W(U(t))\dd t}{\beta} +\Big(\frac{4J_{\mathrm{min}}}{\beta}+2\|\ee'_0\|_{L^2(\R;\R^m)}^2\Big)(t_2-t_1)=:C_1(|t_2-t_1|).
\end{equation}
This implies that $h_x \in L^2((t_1,t_2)\times\R;\R^m)$, and $\| h_x\|^2_{L^2((t_1,t_2)\times\R;\R^m)}\leq C_1(|t_2-t_1|)$. The proof that $h_{xx} \in L^2((t_1,t_2)\times\R;\R^m)$, with $\| h_{xx}\|^2_{L^2((t_1,t_2)\times\R;\R^m)}\leq C_2(|t_2-t_1|):=4\int_\R\mathcal W(U(t))\dd t+\big(4J_{\mathrm{min}}+2\|\ee'_0\|_{L^2(\R;\R^m)}^2\Big)(t_2-t_1)$ is similar.
\end{proof}

Next, we shall establish that the map $u(t,x):=[U(t)](x)$ is a weak solution of \eqref{system}. Given a function $\phi \in C^1_0(\R^2;\R^m)$, we also view it as a map $\Phi\in C^1_0(\R; L^2(\R;\R^m))$, $t\mapsto \Phi(t)$, by setting $[\Phi(t)](x):=\phi(t,x)$.
For every $\lambda \in \R$, it is clear that
\begin{equation}\label{var1}
\mathcal J_\R(U)\leq\mathcal J_\R(U+\lambda\Phi),
\end{equation}
and
\begin{subequations}\label{var33}
\begin{equation}
\frac{\dd}{\dd\lambda}\Big|_{\lambda=0}\int_\R\frac{1}{2}\|U'(t)+\lambda \Phi'(t)\|_{L^2(\R;\R^m)}^2\dd t=\int_\R\langle U'(t),\Phi'(t)\rangle_{L^2(\R;\R^m)}\dd t,
\end{equation}
\begin{equation}
\frac{\dd}{\dd\lambda}\Big|_{\lambda=0}\int_\R\frac{1}{2}\|U''(t)+\lambda \Phi''(t)\|_{L^2(\R;\R^m)}^2\dd t=\int_\R\langle U''(t),\Phi''(t)\rangle_{L^2(\R;\R^m)}\dd t.
\end{equation}
\end{subequations}
On the other hand, since $\int_\R [\sigma(U'(t))+\mathcal W(U(t))]\dd t<\infty$, it follows that for a.e. $t\in\R$, we have $\sigma(U'(t))+\mathcal W(U(t))<\infty$, and $U(t)\in \ee_0+H^2(\R;\R^m)$ as well as $U'(t)\in H^1(\R;\R^m)$.
Our claim is that
\begin{equation}\label{var22}
\frac{\dd}{\dd\lambda}\Big|_{\lambda=0}\int_\R[\sigma (U'(t)+\lambda\Phi'(t))+\mathcal W(U(t)+\lambda\Phi(t))]\dd t=\int_{\R}\psi(t)\dd t,
\end{equation}
where 
\begin{equation*}
\psi(t):=\int_\R\big[2\frac{\dd [U'(t)]}{\dd x}\cdot \frac{\partial^2\phi(t,x)}{\partial t \partial x} +\frac{\dd^2 [U(t)]}{\dd x^2}\cdot \frac{\partial^2\phi(t,x)}{\partial x^2}+\beta \frac{\dd [U(t)]}{\dd x}\cdot \frac{\partial\phi(t,x)}{\partial x}+\nabla W([U(t)](x))\cdot\phi(t,x)\big]\dd x.
\end{equation*}
Indeed, we first notice that for every $\lambda\neq 0$, the functions
$\psi_\lambda(t):=\frac{1}{\lambda}[\sigma (U'(t)+\lambda\Phi'(t))+\mathcal W(U(t)+\lambda \Phi(t))-\sigma (U'(t))-\mathcal W(U(t))]$ are defined a.e. Moreover, we can see that $\psi_\lambda(t)$ is equal to
\begin{multline*}
\int_\R\Big[2\frac{\dd [U'(t)]}{\dd x}\cdot \frac{\partial^2\phi(t,x)}{\partial t\partial x}+\frac{\dd^2 [U(t)]}{\dd x^2}\cdot \frac{\partial^2\phi(t,x)}{\partial x^2}+\beta \frac{\dd [U(t)]}{\dd x}\cdot \frac{\partial\phi(t,x)}{\partial x} \\
+\lambda \Big|\frac{\partial^2\phi(t,x)}{\partial t\partial x}\Big|^2+\frac{\lambda}{2} \Big|\frac{\partial^2\phi(t,x)}{\partial x^2}\Big|^2+\frac{\beta\lambda}{2} \Big|\frac{\partial\phi(t,x)}{\partial x}\Big|^2+\nabla W([U(t)](x)+c_\lambda(t,x)\lambda \phi(t,x))\cdot\phi(t,x)\Big]\dd x,
\end{multline*}
with $0\leq c_\lambda(t,x)\leq 1$. As a consequence, we obtain $\lim_{\lambda\to 0}\psi_\lambda(t)=\psi(t)$ for a.e. $t\in\R$. Finally, setting $\kappa=\sup\{|\nabla W(v)|: |v|\leq \|u\|_{L^\infty(\supp \phi;\R^m)}+\|\phi\|_{L^\infty(\R^2;\R^m)}\}$, there is 
an integrable function 
\begin{equation*}
\Psi(t)=\sigma(U'(t))+\mathcal W(U(t))+\int_\R\Big[2\Big|\frac{\partial^2\phi(t,x)}{\partial t\partial x}\Big|^2+\Big|\frac{\partial^2\phi(t,x)}{\partial x^2}\Big|^2+\beta \Big|\frac{\partial\phi(t,x)}{\partial x}\Big|^2+\kappa |\phi(t,x)|\Big]\dd x
\end{equation*}
such that $|\psi_\lambda(t)|\leq \Psi(t)$ holds a.e., when $|\lambda|\leq 1$. Thus, we deduce \eqref{var22} by dominated convergence. 

Now, we gather the previous results to conclude.
In view of \eqref{var1}, \eqref{var33} and \eqref{var22}, the minimizer $U$ satisfies the Euler-Lagrange equation
\begin{equation}\label{euler1}
\int_{\R}(\langle U''(t),\Phi''(t)\rangle_{L^2(\R;\R^m)}+\beta\langle U'(t),\Phi'(t)\rangle_{L^2(\R;\R^m)}+\psi(t))\dd t=0.
\end{equation}
 which is equivalent to
\begin{equation}\label{euler2}
\int_{\R^2}(\nabla^2 u\cdot \nabla^2\phi+\beta\nabla u\cdot \nabla\phi +\nabla W(u)\cdot \phi)=0.
\end{equation}
By elliptic regularity, it follows that $u\in C^4(\R^2;\R^m)$ is a classical solution of \eqref{system}. On the other hand, it is clear in view of Lemma \ref{lem1aa} that \eqref{lay1baa} holds. Thus to complete the proof of Theorem \ref{connh2}, it remains to show \eqref{lay2baa}. Let us first establish the uniform continuity of $u$ in the strips $[t_1,t_2]\times\R$ (with $[t_1,t_2]\subset\R$). To see this, we shall consider an arbitrary disc $D$ of radius $1$ included in the strip $[t_1,t_2]\times\R$, and check that for such discs, $\|u\|_{H^2(D;\R^m)}$ is uniformly bounded. Indeed, we notice that $\|u\|_{L^2(D;\R^m)}$ is uniformly bounded (independently of $D$), since the function $\R\ni t\mapsto \|u(t,\cdot)-\ee_0(\cdot)\|_{L^2(\R;\R^m)}$ is continuous. Next, in view of the $L^2$ bounds obtained in Lemma \ref{sobolev} for the first and second derivatives of $u$, we deduce our claim. To prove \eqref{lay2baa}, assume by contradiction the existence of a sequence $(s_k,x_k)$ such that $\lim_{k\to\infty}x_k=\infty$, $s_k\in[t_1,t_2]$, and $|u(s_k,x_k)-a^+|>\epsilon>0$. As a consequence of the uniform continuity of $u$, we can construct a sequence of disjoint discs of fixed radius, centered at $(s_k,x_k)$, over which $W(u)$ is bounded uniformly away from zero. This clearly violates the finiteness of $E_{[t_1,t_2]\times\R}(u)=\mathcal J_{[t_1,t_2]}(U)+J_{\mathrm{min}}(t_2-t_1)$. 
%To prove \eqref{lay1b}, assume by contradiction the existence of a sequence $t_k$ such that $\lim_{k\to\infty}t_k=\infty$, and $d_{\mathcal H}(U(t_k),F^+)>2\epsilon>0$. Since $\R\ni t\mapsto U(t)\in \mathcal H$ is uniformly continuous, we can construct a sequence of disjoint intervals $
%[t_k-\eta, t_k+\eta]$ of fixed length over which $d_{\mathcal H}(U(t),F^+)>\epsilon>0$, and $\mathcal W(U(t))$ is bounded uniformly away from zero (cf. Lemma \ref{lem1w} (ii)). This again violates the finiteness of $\mathcal J_{\R}(U)$. Finally, the equipartition property (iii) is established as in Theorem \ref{connh}, and (iv) follows from \eqref{var1}, since $E_{[\alpha,\beta]\times \R}(u+\phi)=\mathcal J_{[\alpha,\beta]}(U+\Phi)+(\beta-\alpha)J_{\mathrm{min}}$, if $\supp \phi\subset (\alpha,\beta)\times\R$.
\end{proof}

\section*{Acknowledgments}
This research is supported by REA - Research Executive Agency - Marie Sk{\l}odowska-Curie Program (Individual Fellowship 2018) under Grant No. 832332, by the Basque Government through the BERC 2018-2021 program, by the Ministry of Science, Innovation and Universities: BCAM Severo Ochoa accreditation SEV-2017-0718, by project MTM2017-82184- R funded by (AEI/FEDER, UE) and acronym ``DESFLU'', and by the National Science Centre, Poland (Grant No. 2017/26/E/ST1/00817). 

\bibliographystyle{plain}

\end{document}